\documentclass[12pt]{amsart}
\usepackage{amsmath,amsthm,amsfonts,amssymb,latexsym}
\usepackage[pagebackref]{hyperref}
\usepackage{enumerate}
\usepackage[shortlabels]{enumitem}
\usepackage{mathabx}

\newcommand{\FF}{{\mathbb F}}
\newcommand{\KK}{{\mathbb K}}
\newcommand{\QQ}{{\mathbb Q}}
\newcommand{\ZZ}{{\mathbb Z}}

\newcommand{\bC}{{\mathbf{C}}}
\newcommand{\bG}{{\mathbf{G}}}
\newcommand{\bH}{{\mathbf{H}}}
\newcommand{\bL}{{\mathbf{L}}}
\newcommand{\bN}{{\mathbf{N}}}
\newcommand{\bP}{{\mathbf{P}}}
\newcommand{\bT}{{\mathbf{T}}}
\newcommand{\bU}{{\mathbf{U}}}
\newcommand{\bZ}{{\mathbf{Z}}}

\newcommand{\cG}{{\mathcal{G}}}
\newcommand{\cE}{{\mathcal{E}}}
\newcommand{\cJ}{{\mathcal{J}}}
\newcommand{\cS}{{\mathcal{S}}}

\newcommand{\Gal}{{\operatorname{Gal}}}
\newcommand{\Ind}{{\operatorname{Ind}}}
\newcommand{\Infl}{{\operatorname{infl}}}
\newcommand{\Irr}{{\operatorname{Irr}}}
\newcommand{\lcm}{{\operatorname{lcm}}}
\newcommand{\ord}{{\operatorname{ord}}}
\newcommand{\Res}{{\operatorname{Res}}}
\newcommand{\Syl}{{\operatorname{Syl}}}

\newcommand{\GL}{{\operatorname{GL}}}
\newcommand{\PGL}{{\operatorname{PGL}}}
\newcommand{\SL}{{\operatorname{SL}}}
\newcommand{\GU}{{\operatorname{GU}}}

\let\al=\alpha
\let\eps=\epsilon
\let\ze=\zeta

\newcommand{\tw}[1]{{}^{#1}\!}
\newcommand{\wt}{\widetilde}
\newcommand{\lev}{\mathrm{\mathbf{lev}}}

\newtheorem{theorem}{Theorem}[section]
\newtheorem{lemma}[theorem]{Lemma}

\newtheorem{corollary}[theorem]{Corollary}

\newtheorem{theorema}{Theorem}

\newtheorem{conjecturea}[theorema]{Conjecture}
\newtheorem{corollarya}[theorema]{Corollary}

\theoremstyle{definition}

\newtheorem{notation}[theorem]{Notation}
\newtheorem{remark}[theorem]{Remark}

\newtheorem*{example*}{Example}

\numberwithin{equation}{section}
\makeatletter \@namedef{subjclassname@2020}{\textup{2020}
	Mathematics Subject Classification} \makeatother

\begin{document}

\title[The $p$-rationality of Deligne--Lusztig characters]
{On the $p$-rationality of Deligne--Lusztig characters}

\author[Nguyen N. Hung]{Nguyen N. Hung}
\address{Department of Mathematics, The University of Akron, Akron,
OH 44325, USA}
\email{hungnguyen@uakron.edu}

\thanks{I am indebted to Gunter Malle for several helpful conversations leading to the proofs of Theorems~\ref{thm:main} and \ref{thm:p=r}. The proof of Theorem~\ref{thm:p=r} is essentially due to him. The author gratefully acknowledges support from the AMS--Simons Research Enhancement Grant (AWD-000167 AMS)}



\begin{abstract}
Among finite simple groups, character values of alternating and sporadic groups
have relatively low irrationality at any prime $p$, whereas those of simple
groups of Lie type can have arbitrarily high $p$-irrationality. We provide
concrete evidence supporting this phenomenon. In particular, we show that if
$\chi:=R_{\bT}^{\bG}(\theta)$ is a Deligne--Lusztig character of a finite
reductive group $\bG^F$, with $\theta$ an irreducible character of a maximal
torus $\bT^F$, and if $\chi$ has degree prime to $p$, then the so-called
$p$-rationality level of $\chi$ coincides precisely with that of $\theta$.
We present further evidence suggesting that Lusztig induction
preserves $p$-rationality for characters of $p'$-degree.
\end{abstract}

\maketitle



\section{Introduction}

Studying the rationality of character values of finite groups is a classical
problem in group representation theory. Equally significant is the study of
rationality with respect to a fixed prime $p$, commonly referred to as the
$p$-rationality of character values.

In this paper, we are concerned with the $p$-rationality of Deligne--Lusztig
characters and, more generally, Lusztig induced characters of finite groups of
Lie type. By such a group we mean the group of fixed points $\bG^F$, where
$\bG$ is a connected reductive linear algebraic group defined over an algebraic
closure of a finite field of characteristic $r>0$, and $F:\bG\to\bG$ is a
Steinberg endomorphism. Deligne--Lusztig characters are virtual characters of
$\bG^F$ constructed via $l$-adic cohomology of certain varieties, and they
provide a systematic framework that
classifies (and in many cases explicitly constructs) all irreducible complex
representations of these groups. For the definition and significance of
Deligne--Lusztig characters in the representation theory of finite groups of
Lie type, we refer the reader to \cite[Chapter~7]{Carter85} or
\cite[Chapters~9-11]{Digne-Michel}.

As usual, let $\Irr(G)$ denote the set of irreducible complex characters of a
finite group $G$. Let $\chi\in\ZZ\Irr(G)$ be a virtual character of~$G$,
and let $\QQ(\chi)$ be its field of values. The \emph{conductor} of~$\chi$,
denoted $c(\chi)$, is the smallest positive integer~$n$ such that $\QQ(\chi)$
is contained in the $n$-th cyclotomic field $\QQ(\ze_n)$. The $p$-part of
$c(\chi)$ measures the extent to which the values of $\chi$ fail to be
$p$-rational. Accordingly, the $p$-adic valuation $\nu_p\bigl(c(\chi)\bigr)$
is called the \emph{$p$-rationality level} of $\chi$ and will be denoted by
$\lev(\chi)$ for simplicity, where the prime $p$ will be clear from the
context.

\subsection{Motivation} The notion of $p$-rationality, while interesting in
its own right, arises naturally in the context of several well-known problems
in the field. One of them is the McKay--Navarro conjecture \cite{Navarro04},
a refinement of the McKay conjecture \cite{McKay} whose proof was recently completed by
Cabanes and Sp\"ath \cite{CS24}. Ruhstorfer and Schaeffer Fry \cite{RS25},
building on
the reduction to finite simple groups of the McKay--Navarro conjecture by
Navarro, Sp\"{a}th, and Vallejo \cite{NSV20}, established a major consequence:
for a finite group $G$ and a Sylow $p$-subgroup $P$, there exists a bijection
between the sets of $p'$-degree irreducible characters of $G$ and of $\bN_G(P)$
that preserves the $p$-rationality level.

Another remarkable problem is a conjecture of Navarro and Tiep
\cite{Navarro-Tiep21} concerning the field of values of Sylow restrictions.
It asserts that if $\chi$ is a $p'$-degree irreducible character of $G$ and
$P\in\Syl_p(G)$, then
$\QQ(\ze_p)(\Res_P^G(\chi))=\QQ(\ze_{\nu_p(c(\chi))})$. In the language of
$p$-rationality, this is essentially equivalent to
$\lev(\chi)=\lev(\Res_P^G(\chi))$ whenever $\lev(\chi)\ge 2$.

One of the most viable to
many problems in group representation theory is to reduce them to the case of finite simple groups.
In some cases, this is the only approach currently known,
including for the problems mentioned above. Among finite
simple groups, contrasting phenomena occur. Characters of alternating groups
are almost $p$-rational for all primes $p$ (see \cite[\S2.5]{JK81}), while
those of sporadic simple groups also have relatively low $p$-irrationality
(\cite{Atlas,GAP}). On the other hand, characters of simple groups of
Lie type, the most challenging case, usually have higher
$p$-irrationality.

The main motivation for this work is to make concrete this phenomenon. At the same
time, we would like to understand the $p$-rationality
of Deligne--Lusztig characters, which serve as building blocks for all complex
irreducible characters of finite reductive groups. More generally, we study
how Lusztig induction affects the $p$-rationality of character values.

\subsection{Results}
Our first result shows that the $p$-rationality of a Deligne--Lusztig
character is essentially preserved under induction; more precisely, it
coincides with that of the character from which it is induced. Throughout,
we say that $\chi$ is \emph{almost $p$-rational} if $\lev(\chi)\le1$.

\begin{theorema}   \label{thm:main}
 Let $(\bG,F)$, $r$, and $p$ be as above. Let $\bT\le\bG$ be an $F$-stable
 maximal torus and $\theta\in \Irr(\bT^F)$. Let
 \[\chi:=R_{\bT}^{\bG}(\theta)\in\ZZ\Irr(\bG^F)\]
 be the Deligne--Lusztig character of $\bG^F$ associated to $\bT$ and $\theta$.
 Suppose that $\chi$ has $p'$-degree. Then
 \[ \lev(\chi) = \lev(\theta) \]
 whenever $\max\{\lev(\theta),\lev(\chi)\}\ge 2$. Otherwise, both $\chi$ and
 $\theta$ are almost $p$-rational.
\end{theorema}

We present an example illustrating Theorem~\ref{thm:main}.

\begin{example*}
The linear group $G=\mathrm{GL}_2(q)$, where $q$ is a prime power,
has a conjugacy classes of split maximal tori isomorphic to
$\mathbb{F}_q^\times \times \mathbb{F}_q^\times$, with a
representative $T=\{\mathrm{diag}(x,y): x,y\in
\mathbb{F}_q^\times\}$. Each $\theta=\lambda_1\boxtimes\lambda_2\in
\Irr(T)$ with $\lambda_1\neq \lambda_2\in \Irr(\mathbb{F}_q^\times)$
has trivial stabilizer in the Weyl group $W(G,T)\cong C_2$, and
consequently the associated Deligne--Lusztig character
$\chi:=R_{T}^G(\theta)$ is an irreducible character of $G$ (of
degree $q+1$). Fix a generator $x$ of $\mathbb{F}_q^\times$ and let
$\zeta:=\zeta_{q-1}$ be a primitive $(q-1)$-th root of unity.
Suppose that $\lambda_i(x)=\zeta^{m_i}$ for $i=1,2$, with $m_i\in
\mathbb{Z}_{\ge 0}$. Then the set of values of $\theta$ is
$\mathcal{V}_\theta=\{\zeta^{m_1 a+m_2 b}: a,b\in \mathbb{Z}\}$. On
the other hand, the set of values of $\chi$ is
$\mathcal{V}_\chi=\{0,\ \zeta^{(m_1+m_2)a},\
(q+1)\zeta^{(m_1+m_2)a},\ \zeta^{m_1 b+m_2 c}+\zeta^{m_1 c+m_2 b} :
a,b,c\in \mathbb{Z}\}$ (see \cite[Table~2.5]{GM20}). In this case,
Theorem~\ref{thm:main} claims that if a prime $p$ does not divide
$q+1$, then the $p$-parts of the conductors of the two sets
$\mathcal{V}_\theta$ and $\mathcal{V}_\chi$ (of cyclotomic integers)
are the same, except possibly in the case where one is $1$ and the
other is $p$.
\end{example*}

One immediate consequence of Theorem~\ref{thm:main} is that the
Deligne--Lusztig character under consideration contains an irreducible
constituent whose $p$-rationality level is at least that of $\theta$. This
allows us to prove the existence of irreducible characters with high
$p$-irrationality in the relevant \emph{Lusztig series}.

Let $\bG^\ast$ be a connected reductive algebraic group and
$F^\ast:\bG^\ast\to\bG^\ast$ be a Steinberg map such that $(\bG,F)$ and
$(\bG^\ast,F^\ast)$ are in duality. For a semisimple element
$s\in {\bG^\ast}^{F^\ast}$, the Lusztig series $\cE(\bG^F,s)$ associated
with~$s$ is the set of irreducible constituents of the Deligne--Lusztig
characters $R_{\bT}^{\bG}(\theta)$ corresponding to pairs $(\bT,\theta)$ in
the geometric conjugacy class determined by $s$; see
Section~\ref{sec:main-theorem} for details.

\begin{corollarya}   \label{cor:main2}
 Let $s \in {\bG^\ast}^{F^\ast}$ be a semisimple element with
 $\nu_p(\ord(s)) \ge 2$. Suppose that $p$ does not divide
 $|{\bG^\ast}^{F^\ast} : {\bT^\ast}^{F^\ast}|_{r'}$ for some $F^\ast$-stable
 maximal torus $\bT^\ast \subseteq \bG^\ast$ containing $s$. Then the Lusztig
 series $\cE(\bG^F,s)$ contains a character whose $p$-rationality level is at
 least $\nu_p(\ord(s))$.
\end{corollarya}

Can Theorem~\ref{thm:main} be extended to more general Lusztig induced
characters $R_{\bL \le \bP}^{\bG}(\psi)$ where $\bL$ is an $F$-stable Levi
subgroup of a parabolic subgroup $\bP$ of $\bG$? Some further cases we prove
in Section~\ref{sec:ConjC} have convinced us to put forward the following.

\begin{conjecturea}   \label{conj:Lusztig-induction}
 Let $\bG,F$, $r$, and $p$ be as above. Let $\bL$ an $F$-stable Levi subgroup
 of a parabolic subgroup $\bP$ of $\bG$, and $\psi \in\Irr(\bL^F)$. Suppose
 that the Lusztig induced character
 $R_{\bL \le\bP}^{\bG}(\psi)\in\ZZ\Irr(\bG^F)$ has $p'$-degree. Then
 \[ \lev\left(R_{\bL \le \bP}^{\bG}(\psi)\right) = \lev(\psi)\]
 whenever $\max\{\lev(\psi),\lev(R_{\bL \le \bP}^{\bG}(\psi))\}\ge2$.
 Otherwise, both $\psi$ and $R_{\bL \le \bP}^{\bG}(\psi)$ are almost
 $p$-rational.
\end{conjecturea}

We note that Conjecture~\ref{conj:Lusztig-induction} is new even in the case
where the parabolic subgroup $\bP$ is $F$-stable. In this situation, Lusztig
induction reduces to Harish-Chandra induction, and $R_{\bL\le \bP}^\bG(\psi)$
is obtained from $\psi$ by inflating it to $\bP^F$ and then inducing to
$\bG^F$: $R_{\bL \le\bP}^\bG(\psi)=
\Ind^{\bG^F}_{\bP^F}(\Infl_{\bL^F}^{\bP^F}(\psi))$. Let
$\Psi:=\Infl_{\bL^F}^{\bP^F}(\psi)$. The
$F$-stable-parabolic case of the conjecture becomes
$\lev\bigl(\Ind^{\bG^F}_{\bP^F}(\Psi)\bigr)=\lev(\Psi)$. This
equality naturally raises a question: does ordinary induction
preserve $p$-rationality when the induced character has $p'$-degree? 
Unfortunately, we do not know the answer at this time.

The paper is organized as follows. In Section~\ref{sec:invariant},
we prove the necessary preparatory lemmas about the so-called
$\ell$-invariant. They are then used in
Section~\ref{sec:main-theorem} to establish Theorem~\ref{thm:main}
and Corollary~\ref{cor:main2}. In the final Section~\ref{sec:ConjC},
we verify Conjecture~\ref{conj:Lusztig-induction} in the case where
$p$ is the defining characteristic of $\bG$, as well as when $\bL^F$
is $p$-solvable and $\bP$ is $F$-stable. We also take the
opportunity to prove Navarro--Tiep's conjecture
\cite[Conjecture~C]{Navarro-Tiep21}, mentioned above, for finite
reductive groups in the same characteristic $p$.

\section{The $\ell$-invariant}\label{sec:invariant}

We study the relationship between the $p$-rationality of a character of a Levi
subgroup and that of its Lusztig induced character via an invariant of
characters, which we call the \emph{$\ell$-invariant}. This invariant first
appeared in the work of Isaacs and Navarro \cite{IN} on the $p$-rationality of
Sylow restrictions in $p$-solvable groups.

We adopt notation standard in the representation theory of finite groups.
In particular, for finite groups $H \le G$, we use $\Ind_H^G(\psi)$ for the
character of $G$ induced from a character $\psi$ of $H$, and $\Res_H^G(\chi)$
for the restriction to $H$ of a character $\chi$ of $G$.

Throughout, we fix a prime $p$. The notation $\lev(\cdot)$ always refer to the
$p$-rationality level of the relevant character. For an abelian extension $\KK$
of $\QQ$, the \emph{conductor} of $\KK$, denoted $c(\KK)$, is the smallest
positive integer $n$ such that $\KK\le\QQ(\ze_n)$. For a virtual character
$\Xi$ of a finite group $G$, recall that
\[c(\Xi):=c(\QQ(\Xi))\quad\text{and}\quad\lev(\Xi):=\nu_p\bigl(c(\Xi)\bigr),\]
where $\nu_p(k):=\log_p(k_p)$ (for $k\in \ZZ_{+}$) denotes the $p$-adic valuation.

We define
\[
\ell(\Xi):=\max\left\{i\in\ZZ_{\geq 0}\mid \Delta_i(\Xi)(1) \notequiv 0
\bmod p\right\},
\]
where, for every $i\in\ZZ_{\ge0}$, the character $\Delta_i(\Xi)$ is
the sum of the irreducible constituents (counting multiplicities) of
$\Xi$ of $p$-rationality level $i$. That is,
\[ \Delta_i(\Xi):=\sum_{\chi\in\Irr(G); \lev(\chi)=i} [\chi,\Xi]\, \chi. \]
If $\Delta_i(\Xi)(1)\equiv 0 \pmod p$ for all $i\in\ZZ_{\ge 0}$, we
adopt the convention that $\ell(\Xi)=0$.

We begin with a variation of \cite[Lemma~4.2]{HS26} on the
relationship between the $\ell$-invariant and the $p$-rationality
level.

\begin{lemma}   \label{lem:2}
 Let $\Xi$ be a virtual character of a finite group $G$. Suppose that
 $\ell(\Xi)\geq 2$. Then $\lev(\Xi)\ge \ell(\Xi)$.
\end{lemma}

\begin{proof}
Let $\al := \ell(\Xi) \ge 2$. Write $|G|=p^{a}m$, where $a\in\ZZ_{\ge0}$ and
$m\in\ZZ_{\ge 1}$ is coprime to $p$. Note that the $p$-rationality level of
every irreducible character of $G$ is at most $a$, so $\Delta_i(\Xi)= 0$ for
every $i\ge a$. Therefore $a\ge \al$. Consider the Galois group
$\cG:=\Gal\left(\QQ(\ze_{p^am}) /\QQ(\ze_{p^{\al-1}m})\right)$.
Restriction to $\QQ\left(\ze_{p^a}\right)$ induces an injective homomorphism
from $\cG$ to the group $\Gal\left(\QQ(\ze_{p^a})/\QQ(\ze_{p^{\al-1}})\right)$
of order $p^{a-\al+1}$, and hence $\cG$ is a non-trivial $p$-group.

Assume, for a contradiction, that $\lev(\Xi) \le \al - 1$.
Equivalently, $\Xi$ is fixed by $\cG$. Recall that every irreducible
constituent $\chi$ of $\Delta_{\al}(\Xi)$ has level $\al$, thus
\[ \QQ(\chi) \nsubseteq \QQ\left(\ze_{p^{\al-1}m}\right), \]
so $\chi$ is not invariant under $\cG$. It follows that every $\cG$-orbit of
irreducible constituents of $\Delta_{\al}(\Xi)$ has size a non-trivial power
of $p$. Consequently, $\Delta_{\al}(\Xi)(1)$ is divisible by $p$,
contradicting the definition of $\al$.
\end{proof}

We also require the following result, which improves upon
\cite[Lemma~4.3]{HS26} by removing the hypothesis on $\psi$ being irreducible
of $p'$-degree.

\begin{lemma}   \label{lem:3}
 Let $P\leq M\leq G$ where $P\in \Syl_p(G)$, $\psi$ a virtual character of $M$,
 and $\chi:=\Ind_M^G(\psi)$. Then
 \[\ell\left(\Res^G_P(\chi)\right)=\ell\left(\Res^M_P(\psi)\right)\]
 or else
 $\left\{\ell\left(\Res^G_P(\chi)\right),\ell\left(\Res^M_P(\psi)\right)\right\}\subseteq\{0,1\}$.
\end{lemma}

\begin{proof}
Fix a set of representatives for the double cosets $M\backslash G/P$, and let
$T$ denote the subset consisting of those representatives $x$ for which
$P \subseteq M^x$. The proof of \cite[Lemma~4.3]{HS26} shows that
\[\begin{aligned}
 \Delta_i\!\left(\Res^G_P\!\left(\Ind_M^G(\varphi)\right)\right)(1)
  &\equiv \sum_{t\in T} \Delta_i\left(\Res_P^M(\varphi^t)\right)(1) \pmod p\\
  &\equiv |T|\cdot \Delta_i\!\left(\Res^M_P(\varphi)\right)(1) \pmod p
\end{aligned}\]
for every $\varphi\in\Irr(M)$ and every $i\ge 2$. We deduce that
\[\Delta_i\!\left(\Res^G_P(\chi)\right)(1) \equiv |T|\cdot
\Delta_i\!\left(\Res^M_P(\psi)\right)(1) \pmod p\]
for every $i\ge 2$. Since $p$ does not divide $|T|$ by \cite[Lemma~3.4]{IN},
it follows that $\Delta_i\left(\Res^G_P(\chi)\right)(1)$ is divisible by $p$
if and only if $\Delta_i\left(\Res^M_P(\psi)\right)(1)$ is divisible by~$p$.
This proves the lemma.
\end{proof}

\begin{lemma}   \label{lem:1}
 Let $p$ be a prime, $M \le G$ with $p\nmid |G:M|$, and $P\in\Syl_p(M)$. Let
 $\psi$ be a virtual character of $M$ with $\lev(\psi)\ge 2$, and set
 $\chi:=\Ind_M^G(\psi)$. Then
 \[\lev(\psi)=\lev(\chi)=\lev\left(\Res^G_P(\chi)\right),\]
 provided that one of the following conditions holds:
 \begin{enumerate}[\rm(1)]
  \item $\lev(\psi)=\ell\left(\Res_P^G(\psi)\right)$; or
  \item $\psi$ is linear.
 \end{enumerate}
\end{lemma}

\begin{proof}
Suppose that (1) holds, so $\lev(\psi)=\ell\left(\Res_P^G(\psi)\right)$. By the
assumption on $\lev(\psi)$, we have $\ell\left(\Res_P^G(\psi)\right)\ge 2$.
By Lemma~\ref{lem:3}, it follows that
\[ \ell\left(\Res_P^G(\psi)\right)=\ell\left(\Res_P^G(\chi)\right). \]
On the other hand, Lemma~\ref{lem:2} implies that
\[ \lev\left(\Res_P^G(\chi)\right)\ge \ell\left(\Res_P^G(\chi)\right). \]
Combining these (in)equalities, we obtain
\[ \lev\left(\Res_P^G(\chi)\right)\ge \lev(\psi). \]
However, by the character formula we have
$\QQ\left(\Res_P^G(\chi)\right) \subseteq \QQ(\chi) \subseteq \QQ(\psi)$,
which implies that
\[ \lev(\psi)\ge \lev(\chi)\ge \lev\left(\Res_P^G(\chi)\right). \]
Therefore, all these inequalities must be equalities, and the claim is proved
assuming~(1) holds.

If $\psi\in\Irr(M)$ is linear, then the conductor $c(\psi)$ of $\psi$ is
precisely the order of $\psi$ in the abelian group of linear characters of $M$,
and moreover $\QQ(\psi)=\QQ\left(\ze_{c(\psi)}\right)$. Let
$\al:=\lev(\psi)$, so that $c(\psi)=p^\al m$ for some $m\in\ZZ_{\ge1}$ not
divisible by $p$. Then there exists an element $x\in M$ such that
$c(\psi(x))_p=p^\al$. Writing $x_p$ for the $p$-part of $x$, it follows that
$c(\psi(x_p))=p^\al$. Consequently, $\lev(\psi)=\lev(\psi_P)$. Since it is
clear that $\ell\left(\Res_P^G(\psi)\right)=\lev\left(\Res_P^G(\psi)\right)$,
we conclude that $\lev(\psi)=\ell\left(\Res_P^G(\psi)\right)$, and
hence assumption~(2) implies~(1), so we conclude by the first part.
\end{proof}

\section{Theorem \ref{thm:main} and its consequences}\label{sec:main-theorem}

Our notation on representation theory of finite groups of Lie type
follows \cite{Carter85,GM20}.

Recall from the previous section that $p$ is a fixed prime and that
$\lev(\cdot)$ denotes the $p$-rationality level. In this and the
following sections, we fix the following notation.

\begin{notation}   \label{notation}
\begin{enumerate}[\rm(i)]
 \item $r$ is a prime.
 \item $\bG$ is a connected reductive linear algebraic group defined over an
  algebraic closure of a finite field of characteristic $r$ and $F:\bG\to\bG$
  is an associated Steinberg endomorphism.
 \item $\bT\le\bG$ is an $F$-stable maximal torus and $\theta\in\Irr(\bT^F)$.
 \item $\bL$ is an $F$-stable Levi subgroup of a parabolic subgroup $\bP$ of
  $\bG$, and $\psi\in\Irr(\bL^F)$.
\end{enumerate}
\end{notation}

\subsection{Proof of Theorem~\ref{thm:main}}
Deligne--Lusztig characters $R_{\bT}^{\bG}(\theta)$, roughly speaking, are
virtual characters constructed geometrically via the $l$-adic cohomology of
certain algebraic varieties (called Deligne--Lusztig varieties). We refer the
reader to \cite[\S2.2]{GM20} for the precise definition. The degree of
$R_{\bT}^{\bG}(\theta)$ is
\[ R_{\bT}^{\bG}(\theta)(1)=|\bG^F:\bT^F|_{r'} \]
and general values are given by
\begin{equation}
R_{\bT}^{\bG}(\theta)(g)=\frac{1}{|\bC^\circ_\bG(s)^F|}
\sum_{x\in\bG^F:\, x^{-1}sx\in\bT^F} Q_{x\bT
x^{-1}}^{\bC^\circ_\bG(s)}(u)\theta(x^{-1}sx),
\label{eq:DL-character}
\end{equation}
where $g=su=us\in\bG^F$ with $s$ being semisimple and $u$ unipotent is the
Jordan decomposition of $g$. (See \cite[Theorem~2.2.16]{GM20}.)
Here, $\bC^\circ_\bG(s)$ denotes the connected component of the centralizer
$\bC_\bG(s)$ containing the identity element. Also,
$Q_{x\bT x^{-1}}^{\bC^\circ_\bG(s)}$ is the Green function on the set of
unipotent elements $u$ of $\bC^\circ_\bG(s)^F$, defined as
\[Q_{x\bT x^{-1}}^{\bC^\circ_\bG(s)}(u):=
  R_{x\bT x^{-1}}^{\bC^\circ_\bG(s)}(\mathbf{1}_{x\bT x^{-1}})(u)
\]
for every $x\in\bG^F$ such that $x^{-1}sx\in\bT^F$. In particular,
\[Q_{x\bT x^{-1}}^{\bC^\circ_\bG(s)}(1)
  = R_{x\bT x^{-1}}^{\bC^\circ_\bG(s)}(\mathbf{1}_{x\bT x^{-1}})(1)
  =\pm|\bC^\circ_\bG(s)^F:x\bT^Fx^{-1}|_{r'}.
\]
As $|x\bT^F x^{-1}|_{r'}=|\bT^F|_{r'}=|\bT^F|$ for all relevant $x$,
we then have
\[Q_{x\bT x^{-1}}^{\bC^\circ_\bG(s)}(1)=\pm|\bC^\circ_\bG(s)^F|_{r'}/|\bT^F|.\]

We now prove Theorem~\ref{thm:main}, which is restated.

\begin{theorem}   \label{thm:main-repeated}
 Assume Notation~\ref{notation}. Let
 \[\chi:=R_{\bT}^{\bG}(\theta)\in\ZZ\Irr(\bG^F)\]
 be the Deligne--Lusztig character of $\bG^F$ associated to $\bT$ and $\theta$.
 Suppose that $\chi$ has $p'$-degree. Then
 \[ \lev(\chi) = \lev(\theta) \]
 whenever $\max\{\lev(\theta),\lev(\chi)\}\ge 2$. Otherwise, both $\chi$ and
 $\theta$ are almost $p$-rational.
\end{theorem}

\begin{proof}
Since $\bT^F$ has $r'$-order, every $\theta \in \Irr(\bT^F)$ is
$r$-rational. Note that Green functions take integral values
\cite[Proposition~2.2.5]{GM20}. It follows from the character
formula \eqref{eq:DL-character} that $\QQ(\chi)\subseteq
\QQ(\theta)$. Therefore $\chi$ is $r$-rational, and we are done in
the case $p=r$.

We therefore assume that $p\neq r$. In particular, since
$\chi(1)=|\bG^F:\bT^F|_{r'}$ is coprime to $p$, it follows that
$|\bG^F:\bT^F|$ is also coprime to $p$. Let $P\in \Syl_p(\bT^F)$ (so
$P\in\Syl_p(\bG^F)$). Similar to the case $p=r$, if $\theta$ is
almost $p$-rational, then so is $\chi$, and we are done again. Thus
we may assume that $\lev(\theta)\ge 2$.

Let $s$ be a semisimple element of $\bG^F$. The character formula yields
\[
\chi(s) = \pm \frac{1}{|\bC^\circ_\bG(s)^F|_r}\,\Ind_{\bT^F}^{\bG^F}(\theta)(s).
\]
(See \cite[Proposition~7.5.4]{Carter85}.) Set $\varphi :=
\Ind_{\bT^F}^{\bG^F}(\theta)$. Then
\[ \chi(s) = \pm \frac{1}{|\bC^\circ_\bG(s)^F|_r}\,\varphi(s). \]

If $\cS$ is a set of cyclotomic integers (that is, finite $\ZZ$-linear
combinations of roots of unity), we write $c(\cS)$ for the conductor of the
smallest abelian extension of $\QQ$ containing $\cS$, and set
$\lev(\cS):=\nu_p(c(\cS))$.

Let $S$ denote the set of semisimple elements of $\bG^F$, and set
\[ \KK := \QQ\left(\{\varphi(s)\mid s\in S\}\right). \]
Then
\[ \lev(\chi) \ge \lev\left(\{\chi(s)\mid s\in S\}\right) = \lev(\KK). \]
On the other hand, by Lemma~\ref{lem:1}(ii),
\[ \lev(\theta) = \lev(\varphi) = \lev\left(\Res^{\bG^F}_{P}(\varphi)\right).\]
Since $P\subseteq S$, we have
\[ \QQ\left(\Res^{\bG^F}_{P}(\varphi)\right) \subseteq \KK, \]
and hence
\[ \lev\left(\Res^{\bG^F}_{P}(\varphi)\right) \le \lev(\KK). \]
Thus $\lev(\theta)\le \lev(\chi)$. The reverse inequality follows
immediately from the character formula together with the integrality of the
Green function, and the proof is complete.
\end{proof}

\subsection{Corollary~\ref{cor:main2} and other consequences}
It is well known that there is a natural bijection $\Pi$ between the set of
$\bG^F$-orbits of pairs $(\bT,\theta)$, where $\bT$ is an $F$-stable maximal
torus of $\bG$ and $\theta\in\Irr(\bT^F)$ and the set of
${\bG^\ast}^{F^\ast}$-orbits of pairs $(\bT^\ast,s)$, where $\bT^\ast$ is an
$F^\ast$-stable maximal torus of $\bG^\ast$ and $s\in\bT^{\ast\,F^\ast}$.
Inspecting the proof of \cite[Corollary~2.5.14]{GM20} and using
\cite[Lemma~2.5.7]{GM20}, we observe that this correspondence preserves orders;
that is, if the pairs $(\bT,\theta)$ and $(\bT^\ast,s)$ belong to orbits
correspond to each other under $\Pi$, then
\[ \ord(\theta)=\ord(s). \]

Recall also that the Lusztig series $\cE(\bG^F,s)$ associated to a semisimple
element $s\in{\bG^\ast}^{F^\ast}$ consists of the irreducible characters of
$\bG^F$ that occur as constituents of Deligne--Lusztig characters
$R_{\bT}^{\bG}(\theta)$, where $(\bT,\theta)$ corresponds under $\Pi$ to some
$(\bT^\ast,s)$ with $s\in\bT^\ast$ (see \cite[Definition~2.6.1]{GM20}).

\begin{corollary}   \label{cor:main2-repeated}
 Let $p$ be a prime, and let $s \in {\bG^\ast}^{F^\ast}$ be a semisimple
 element with $\nu_p(\ord(s)) \ge 2$. Suppose that $p$ does not divide
 $|{\bG^\ast}^{F^\ast} : {\bT^\ast}^{F^\ast}|$ for some $F^\ast$-stable
 maximal torus $\bT^\ast \subseteq \bG^\ast$ containing $s$. Then the Lusztig
 series $\cE(\bG^F,s)$ contains a character whose $p$-rationality level is at
 least $\nu_p(\ord(s))$.
\end{corollary}

\begin{proof}
The assumption that $s$ is semisimple implies $p\ne r$.
Let $\bT^\ast\le\bG^\ast$ be an $F^\ast$-stable maximal torus containing~$s$
such that $|{\bG^\ast}^{F^\ast}:{\bT^\ast}^{F^\ast}|$ is coprime to $p$. Let
$(\bT,\theta)$ be a pair in the $\bG^F$-orbit corresponding to the
${\bG^\ast}^{F^\ast}$-orbit of $(\bT^\ast,s)$ under $\Pi$. Then
$\ord(s)=\ord(\theta)$. By our assumption on~$s$, we have
$\nu_p(\ord(\theta))\ge 2$, or equivalently, $\lev(\theta)\ge 2$.

Note that, since $|\bT^F| = |{\bT^\ast}^{F^\ast}|$ by \cite[Lemma~2.5.2]{GM20},
it follows that $p$ does not divide $|\bG^F : \bT^F|$. Let
$\chi := R_{\bT}^{\bG}(\theta)$. Theorem~\ref{thm:main} then implies that
\[ \lev(\chi)=\lev(\theta)=\nu_p(\ord(s)). \]
Consequently, $\chi$ has an irreducible constituent of $p$-rationality level
at least $\nu_p(\ord(s))$, which necessarily lies in the Lusztig series
$\cE(\bG^F,s)$.
\end{proof}

Corollary~\ref{cor:main2-repeated} shows that 
if $\ord(s)$ has maximal
$p$-part among all elements of ${\bG^\ast}^{F^\ast}$ and $p$ does
not divide $|{\bG^\ast}^{F^\ast} : {\bT^\ast}^{F^\ast}|$ for some
$F^\ast$-stable maximal torus $\bT^\ast \subseteq \bG^\ast$
containing $s$ (in particular, Sylow $p$-subgroups of $\bG^F$ are
abelian), then
$\nu_p(\ord(s))$ is precisely the maximal $p$-rationality level of a
character in $\cE(\bG^F,s)$. More generally,
Theorem~\ref{thm:main} enables us to determine the maximal
$p$-rationality level of characters in certain Lusztig series of
some groups of Lie type, such as linear and unitary groups, as
illustrated below.

Let $q$ be a power of a prime $r$. Let $\bG =\GL_n(\overline{\FF}_r)$, and let
$F$ be the Frobenius endomorphism
\[ X = (x_{ij}) \longmapsto X^{(q)} := (x_{ij}^q), \]
or the twisted Frobenius
\[ F : X \longmapsto \bigl({}^t X^{(q)}\bigr)^{-1}. \]
Then $G := \bG^F$ is the general linear group $\GL_n(q)$ in the first case and
the general unitary group $\GU_n(q)$ in the second. For notational convenience,
we write $G = \GL_n(\eps q)$, where $\eps\in \{\pm 1\}$, with $\eps = +1$
corresponding to the linear case and $\eps = -1$ to the unitary case.
Note that the dual group $(\bG^\ast)^{F^\ast}$ is isomorphic to $\bG^F$.

For every semisimple element $s \in (\bG^\ast)^{F^\ast}$, it follows from
\cite[\S4]{Hung-Tiep23} that the field of values of every character in the
Lusztig series $\cE(G,s)$ is contained in $\QQ(\ze_{\ord(s)})$.

\begin{theorem}   \label{thm:GL}
 Let $p$ be a prime, $q$ be a power of a prime $r$, $n\in\ZZ_{\ge2}$, and
 $G\in\{\GL_n(\eps q),\SL_n(\eps q)\}$. Let $s\in G^\ast$ be a semisimple
 element with $\nu_p(\ord(s)) \ge 2$. Then the maximal $p$-rationality level
 among the characters in the Lusztig series $\cE(G,s)$ is $\nu_p(\ord(s))$ in
 each of the following cases:
 \begin{enumerate}[\rm(1)]
  \item $G=\GL_n(\eps q)$ and $p$ does not divide $|G^\ast :T^\ast|$ for
   some maximal torus $T^\ast\subseteq G^\ast$ containing $s$.
  \item $G=\SL_n(\eps q)$ and $p$ does not divide
   $\lcm(q-1,|G^\ast :T^\ast|)$ for some maximal torus
   $T^\ast\subseteq G^\ast$ containing $s$.
 \end{enumerate}
\end{theorem}

\begin{proof}
(1) Let $\chi\in\cE(G,s)$. As noted above, we have
$\QQ(\chi)\le \QQ(\ze_{\ord(s)})$, and hence $\lev(\chi)\le\nu_p(\ord(s))$.
Corollary~\ref{cor:main2-repeated}, on the other hand, guarantees the
existence of a character in $\cE(G,s)$ whose $p$-rationality level is at
least $\nu_p(\ord(s))$.
\smallskip

(2) Note that the dual group of $G=\SL_n(\eps q)$ is $G^\ast=\PGL_n(\eps q)$.
Let $s\in G^\ast$ be a semisimple element satisfying $\nu_p(\ord(s)) \ge 2$,
and suppose that $p$ does not divide $\lcm(q-1,|G^\ast : T^\ast|)$ for
some maximal torus $T^\ast\subseteq G^\ast$ containing $s$. Let
$\wt{s}\in\wt{G}:=\GL_n(\eps q)$ be a preimage of $s$ under the natural
projection from $\wt{G}$ to $G$. Since $\ord(s)$ divides $\ord(\wt{s})$,
which divides $(q-1)\ord(s)$, our assumption on $p$ implies that
\[ \nu_p(\ord(\wt{s}))=\nu_p(\ord(s))\ge 2. \]

The Lusztig series $\cE(G,s)$ consists of irreducible constituents of the
restrictions $\Res^{\wt{G}}_G(\wt{\chi})$, where
$\wt{\chi}\in\cE(\wt{G},\wt{s})$. (See
\cite[Proposition~2.6.16]{GM20}, where this statement is proved for
arbitrary regular embeddings; note that
$\SL_n(\overline{\FF}_r)\to\GL_n(\overline{\FF}_r)$ is such a regular embedding.)
Moreover, if $\chi$ is an irreducible constituent of
$\Res^{\wt{G}}_G(\wt{\chi})$ for some $\wt{\chi}$ with $\lev(\wt{\chi})\ge 2$,
then $\lev(\chi)=\lev(\wt{\chi})$, by \cite[Lemma~2.4(ii)]{Hung22} and the
fact that $p\nmid (q-1)$. The desired result for $\SL_n(\eps q)$ now follows
from the corresponding statement for $\GL_n(\eps q)$.
\end{proof}

\section{Evidence for Conjecture
\ref{conj:Lusztig-induction}}\label{sec:ConjC}

In this section, we provide some evidence in support of
Conjecture~\ref{conj:Lusztig-induction}.

Recall that $\bG$ is a connected reductive algebraic group over
$\overline{\FF}_r$, and $F:\bG\to\bG$ is a Steinberg map. We refer the reader
to \cite[\S3.3]{GM20} for background on Lusztig induced characters
$R_{\bL \le \bP}^{\bG}(\psi)$, where $\bL$ is an $F$-stable Levi subgroup of a
parabolic subgroup $\bP$ of $\bG$, and $\psi \in \Irr(\bL^F)$. We do
want to recall the character formula for Lusztig induction,
which will be used several times (see
\cite[Theorem~3.3.12]{GM20}).
\[\begin{aligned}
R_{\bL}^{\bG}&(\psi)(g)=\\&\frac{1}{|\bL^F||\bC^\circ_\bG(s)^F|}
\sum_{x\in\bG^F:\, x^{-1}sx\in\bL} |\bC^\circ_{x\bL x^{-1}}(s)^F|
\sum_{v\in \bC^\circ_{x\bL x^{-1}}(s)^F_{uni}}    Q_{\bC^\circ_{x\bL
x^{-1}}(s)}^{\bC^\circ_\bG(s)}(u,v^{-1})\psi(x^{-1}svx).\end{aligned}
\]
Here, $g\in\bG^F$ has Jordan decomposition $g=su=us$. Also,
$Q_{\bC^\circ_{x\bL x^{-1}}(s)}^{\bC^\circ_\bG(s)}$ denotes the two-parameter
Green function associated with $\bC^\circ_{x\bL x^{-1}}(s)$ and
$\bC^\circ_\bG(s)$, which is a rational-valued function 
on pairs $(u,v)$ where
$u$ and $v$ are unipotent elements of $\bC^\circ_\bG(s)$ and
$\bC^\circ_{x\bL x^{-1}}(s)$, respectively. See \cite[Definition~3.3.11]{GM20}.

\subsection{The case $p=r$}   \label{sec:casep=r}

In this section we treat the case $p=r$ of the conjecture.
Accordingly, $\lev(\cdot)$ will be the $r$-rationality level.

We aim to prove the following.

\begin{theorem}   \label{thm:p=r}
 Assume Notation~\ref{notation}. Assume furthermore that $\psi\in\Irr(\bL^F)$
 has degree prime to $r$. Consider the Lusztig induced character
 $\chi:=R_{\bL \le \bP}^{\bG}(\psi)\in\ZZ\Irr(\bG^F)$.
 Then exactly one of the following holds.
 \begin{enumerate}[\rm(1)]
  \item Both $\chi$ and $\psi$ are almost $r$-rational; or
  \item $\lev(\chi)=\lev(\psi)=2$.
 \end{enumerate}
 In particular, Conjecture~\ref{conj:Lusztig-induction} holds in the case $p=r$.
\end{theorem}

Note that an $F$-stable Levi subgroup of a parabolic subgroup $\bP$ of $\bG$ is
itself a connected reductive algebraic group defined over $\overline{\FF}_r$.
We then have
\[ \bL = \bZ(\bL)^\circ \cdot [\bL,\bL], \]
where $\bZ(\bL)^\circ$ is the connected component of the center $\bZ(\bL)$
containing the identity element. Note also that $[\bL,\bL]$ is semisimple,
so we suppose that
\[ [\bL,\bL] = \bL_1 \cdot  \cdots \bL_n \]
is a decomposition of $[\bL,\bL]$ as a central product of closed normal
simple subgroups.

Let $\bL_{sc} := [\bL,\bL]_{sc}$ be the simply connected cover of
$[\bL,\bL]$; it is itself a semisimple group of simply connected
type with the same root system as $[\bL,\bL]$. By
\cite[Proposition~1.5.10]{GM20},
\[ \bL_{sc} = \wt{\bL}_1 \times \cdots \times \wt{\bL}_n, \]
where each $\wt{\bL}_i$ is the (simple) simply connected cover of $\bL_i$.

Note that $F$ permutes the simple factors $\bL_1,\dots,\bL_n$, thereby
inducing an action on the index set $\{1,\dots,n\}$. Let
$\cJ\subseteq\{1,\dots,n\}$ be such that $\{\bL_j \mid j\in \cJ\}$ is a
complete set of representatives for the $F$-orbits, and let $l_j$ denote
the length of the $F$-orbit containing $\bL_j$. We then have
\begin{equation}   \label{eq:LscFdirectproduct}
  \bL_{sc}^F \cong \prod_{j \in \cJ} \wt{\bL}_j^{F^{l_j}}.
\end{equation}
(See \cite[Corollary~1.5.16]{GM20}.) Consider the central isogeny
$\pi: \bL_{sc} \to [\bL,\bL]$. Its restriction to $\wt{\bL}_j$ gives the
corresponding isogeny $\wt{\bL}_j\to\bL_j$ for all $j\in\{1,\dots,n\}$. Let
\[ M_j := \pi(\wt{\bL}_j^{F^{l_j}}) \quad \text{for }j\in\cJ. \]
Then
\[ M := \pi(\bL_{sc}^F) \cong \prod_{j \in \cJ} M_j. \]

Note that $M_j$ ($j\in \cJ$) is normal in $(\prod_{k\in \cJ(j)}\bL_k)^F$,
where $\cJ(j)$ is the $F$-orbit of $j$ in $\{1,\dots,n\}$ (see
\cite[1.23]{DL76}). Since $\bL$ is the central product of the $[\bL,\bL]$
with $\bZ(\bL)$ and $[\bL,\bL]$ is the central product of the $\bL_i$, it
follows that $M_j$ is normal in $\bL^F$ for every $j\in\cJ$. Moreover,
$M$ is normal in $\bL^F$ of index prime to $r$.

Irreducible characters of $r'$-degree of finite reductive groups turn out to
have relatively low $r$-rationality level -- at most~$2$, indeed. Recall that
a character $\chi\in\ZZ\Irr(G)$ is said to be \emph{almost $p$-rational}
if its $p$-rationality level is at most~$1$, that is, $\QQ(\chi)\le
\QQ\bigl(\ze_{p|G|_{p'}}\bigr)$. (This generalizes the usual
notion of $p$-rationality, see \cite{Hung-Malle-Maroti21}.)

\begin{lemma}\label{lem:Malle}
 Let $H := \bH^F$, where $\bH$ is a simple algebraic group of simply connected
 type defined in characteristic~$r$, and $F:\bH\to \bH$ is a Steinberg map.
 Then every irreducible character of $H$ of degree prime to $r$ has
 $r$-rationality level at most $2$. Moreover, every such character is almost
 $r$-rational, except when $H = \tw2B_2(2)$ or $\tw2F_4(2)$.
\end{lemma}

\begin{proof}
Let $\sigma\in\Gal(\QQ^{ab}/\QQ)$ denote the automorphism that fixes all roots
of unity of order prime to $r$ and sends each $r$-power root of unity $\xi$ to
$\xi^{1+r}$. Let $\chi$ be a virtual character of a finite group $G$ of order
divisible by~$r$. We shall use the same notation $\sigma$ for its restriction
to $\QQ(\ze_{|G|})$.

When $r$ is odd, the group
$\Gal(\QQ(\ze_{|G|})/\QQ(\ze_{r|G|_{r'}}))$ is canonically isomorphic to
$\Gal(\QQ(\ze_{|G|_r})/\QQ(\ze_r))$, and therefore is a cyclic $r$-group.
The fixed field of $\sigma$ in $\QQ(\ze_{|G|})$ is $\QQ(\ze_{r|G|_{r'}})$.
It follows that $\chi$ is almost $r$-rational if and only if it is fixed by
$\sigma$. This equivalence does not hold in general when $r=2$.

Now suppose $r=2$. The fixed field of $\sigma$ inside $\QQ(\ze_{|G|})$ is
$\QQ(\ze_{|G|_{2'}})$ when $|G|_2 \le 4$, and $\QQ(\sqrt{-2},\ze_{|G|_{2'}})$
when $|G|_2 \ge 8$. There do exist $\sigma$-fixed characters whose field of
values is $\QQ(\sqrt{-2})$; in particular, their $2$-rationality level is
equal to $3$. However, by the main result of \cite{ILNT}, if $\chi$ is an
irreducible character of odd degree with $\lev(\chi)\ge 2$, then
$\sqrt{-1}\in \QQ(\chi)$. Consequently, for irreducible characters of odd
degree in the case $r=2$, the same conclusion holds: namely, $\chi$ is almost
$r$-rational if and only if it is fixed by $\sigma$.

It was shown in the proof of \cite[Proposition~2.4]{Malle19} that every
irreducible character of $H$ of degree prime to $r$ is $\sigma$-fixed, and
hence almost $r$-rational, except possibly when
\[ H\in\{\tw2B_2(2),\ {}^2G_2(3),\ G_2(2),\ G_2(3),\ \tw2F_4(2),\ F_4(2)\}. \]
For these exceptions, the character tables in \cite{Atlas} show that only
$\tw2B_2(2)$ and $\tw2F_4(2)$ have odd-degree irreducible characters that are
not almost $2$-rational. These characters indeed have $2$-rationality level
exactly $2$, completing the proof.
\end{proof}

\begin{lemma}   \label{lem:p=r}
 Assume the hypotheses of Theorem~\ref{thm:p=r} and suppose $\psi\in
 \Irr_{r'}(\bL^F)$ with $\lev(\psi)\ge 2$. Then $\lev(\psi)=2$. Moreover, the
 following hold:
 \begin{enumerate}[\rm(a)]
  \item $r=2$,
  \item $\bL$ has a simple component, say $\bL_1$, of type $B_2$ or $F_4$,
  \item $\bL_1$ is a direct factor of $\bL$,
  \item $\bL_1$ is $F$-stable,
  \item $\bL_1^F \in \{\tw2B_2(2), \tw2F_4(2)\}$, and
  \item $\Res^{\bL^F}_{\bL_1^F}(\psi)$ is an integer multiple of an
   irreducible character of level $2$ of $\bL_1^F$.
 \end{enumerate}
\end{lemma}

\begin{proof}
We keep the notation above.
Let $\phi \in \Irr(M)$ be a character lying under $\psi$, and let $\varphi$
denote its inflation to $\bL_{sc}^F$. If $\lev(\phi) \le1$, so that $\phi$ is
almost $r$-rational, then $\phi$ is fixed by the Galois automorphism $\sigma$
considered in Lemma~\ref{lem:Malle}. By \cite[Lemma~4.2]{Navarro-Tiep21},
it
would follow that $\psi$ is also $\sigma$-fixed, a contradiction. Hence
\[ \lev(\varphi)=\lev(\phi)\ge 2. \]

Recall from \eqref{eq:LscFdirectproduct} that $\bL_{sc}^F$ is the
direct product of the groups $M_j:=\wt{\bL}_j^{F^{l_j}}$
($j\in \cJ$). Accordingly, $\varphi$ is an exterior tensor product of
characters, say $\varphi_j\in\Irr(M_j)$:
\[ \varphi=\bigboxtimes_{j\in\cJ}\varphi_j. \]
Thus
\[ \lev(\varphi)=\max\{\lev(\varphi_j)\mid j\in \cJ\}. \]
By Lemma~\ref{lem:Malle}, we have $\lev(\varphi_j)\le 2$, and hence
$\lev(\varphi)\le 2$. Combined with the previous paragraph, this yields
$\lev(\varphi)=2$, and therefore $\lev(\phi)=2$. Applying
\cite[Lemma~2.4]{Hung22}, we obtain $\lev(\psi)=\lev(\phi)=2$, which
proves the first part of the lemma.

Clearly $\lev(\psi) = 2$ occurs if and only if $\lev(\varphi)=2$,
which means that $\lev(\varphi_j)=2$ for some $j\in \cJ$. Without
loss, we assume that
\[ \lev(\varphi_1)=2. \]
Again by Lemma~\ref{lem:Malle}, we then have
$\wt{\bL}_1^{F^{\ell_1}}\in \{\tw2B_2(2), \tw2F_4(2)\}$.
In particular, this forces $r=2$ and $\ell_1=1$. Furthermore,
$\wt{\bL}_1$ is of type $F_4$ or $B_2$. Note that $F_4$ has
only one isogeny type in all characteristics. While $B_2$ has two
isogeny types, in even characteristic, the finite groups of their
$F$-fixed points are isomorphic. Hence
$\bL_1^F=\wt{\bL}_1^{F^{\ell_1}}\in \{\tw2B_2(2), \tw2F_4(2)\}$.

Note that the simple component $\bL_1$, being of type $F_4$ or $B_2$ in
characteristic $2$, has trivial center (see \cite[Theorem~1.12.5]{GLS98}).
Therefore $\bZ(\bL)$ will split off as a direct factor of $\bL$. Passing to
the finite groups of $F$-fixed points, we have that $\bL_1^F$ is a direct
factor of $\bL^F$. Consequently, $\Res^{\bL^F}_{\bL_1^F}(\psi)$ is an integer
multiple of $\varphi_1$, which, by our choice of $\varphi_1$, has level $2$.
\end{proof}

In the following, we continue to use the notation established in
this section.

\begin{proof}[Proof of Theorem~\ref{thm:p=r}]
If $\lev(\psi) \le 1$, then both $\chi$ and $\psi$ are almost $r$-rational,
as in (1). Hence, we assume $\lev(\psi) \ge 2$. By Lemma~\ref{lem:p=r}, we
have $r = 2$, $\lev(\psi) = 2$, and $\bL$ has a simple component of type $B_2$
or $F_4$, denoted $\bL_1$ as in Lemma~\ref{lem:p=r}. Moreover,
$\bL_1^F \in \{\tw2B_2(2),\tw2F_4(2)\}$. Since $\bL_1$ is a direct factor of
$\bL$, there exists an $F$-stable decomposition
\[\bL = \bL_1 \times \bL'\]
and hence $\bL^F = \bL_1^F \times \bL'^F$.

As above, let $\varphi_1$ be an irreducible character of $\bL_1^F$ lying under
$\psi$ with $\lev(\varphi_1) = 2$. Then we can write
\[\psi = \varphi_1 \boxtimes \psi'\qquad\text{for some $\psi' \in
   \Irr(\bL'^F)$}.\]
\medskip

(A) Consider the case where $\bL_1$ is of type $F_4$. Note that the Dynkin
diagram of type $F_4$ cannot appear as a proper subdiagram of any other
irreducible Dynkin diagram. It follows that $\bL_1$ is a component of $\bG$.
As with $\bL_1$ in $\bL$, this $F_4$ component splits off as a direct factor
in $\bG$:
\[ \bG = \bL_1 \times \bG', \]
where $\bG'$ is $F$-stable and contains $\bL'$ as a Levi subgroup. Hence,
\[
\chi := R_{\bL}^{\bG}(\psi) = R_{\bL}^{\bG}(\varphi_1 \boxtimes \psi')
= R_{\bL_1}^{\bL_1}(\varphi_1) \boxtimes R_{\bL'}^{\bG'}(\psi') =
\varphi_1 \boxtimes R_{\bL'}^{\bG'}(\psi').
\]

We then have
\[ \QQ(\chi) = \QQ\left(\varphi_1, R_{\bL'}^{\bG'}(\psi')\right), \]
which implies that $\lev(\chi) \ge \lev(\varphi_1) = 2$. Since
$\lev(\chi)\le\lev(\psi) (= 2)$, by the character formula for Lusztig
induction mentioned above, we conclude that
\[ \lev(\chi) = \lev(\psi) = 2, \]
as claimed in (2) of the theorem.
\medskip

(B) Next, consider the case where $\bL_1$ is of type $B_2$. Then $\bL_1$ lies
in a simple component, say $\bG_1$, of $\bG$, which is of type $B_2$ or $F_4$.
As before, $\bG_1$ splits off as a direct factor of $\bG$. If $\bG_1$ is of
type $B_2$, the argument proceeds similarly as in (A). Hence, we may assume
that $\bG_1$ is of type $F_4$.

As above we have $\bG=\bG_1\times \bG'$ where $\bG'$ is $F$-stable having
$\bL'$ as a Levi subgroup. Therefore it suffices to assume that $\bG=\bG_1$ is
simple of type $F_4$. Now $\bL_1$ must be the only simple component of $\bL$.
It follows by Lemma~\ref{lem:p=r} that $\bG^F=\tw2F_4(2)$ and
$\bL_1^F=[\bL,\bL]^F=\tw2B_2(2)$.

Note that a Levi subgroup of $\tw2F_4(q^2)$ ($q^2=2^{2m+1}$ for
$m\in\ZZ_{\ge0}$) that contains a factor $\tw2B_2(q^2)$ is isomorphic to one
of the following:
\[ \tw2B_2(q^2)\times C_{q^2-1}, \tw2B_2(q^2)\times C_{q^2+\sqrt{2}q+1},
  \text{ or } \tw2B_2(q^2)\times C_{q^2-\sqrt{2}q+1}. \]
(See, for instance, \cite[Proposition~1.3]{Malle91}.) Note also that for all
$q^2 = 2^{2m+1}$, the Suzuki group $\tw2B_2(q^2)$ has exactly two
non-$2$-rational irreducible characters, both with field of values
$\QQ(\sqrt{-1})$ (see \cite{Suzuki62}). These are unipotent characters,
denoted by \[\lambda_1:=\tw2B_2[1,3]\quad \text{and}\quad
\lambda_2:=\tw2B_2[1,5]\] in the notation of \cite{Chevie} (and by $\tw2B_2[a]$
and $\tw2B_2[b]$ in \cite[p.~488]{Carter85}). Lusztig induction of unipotent
characters of Levi subgroups of $\tw2F_4(q^2)$ is known by \cite{BMM} and
available in \cite{GAP} using the \cite{Chevie} package.
\medskip

(i) The Levi subgroup $\bL^F\cong \tw2B_2(q^2)\times C_{q^2-1}$ is split, that
is, it is contained in an $F$-stable parabolic subgroup of $\bG$. Hence Lusztig
induction coincides with Harish-Chandra induction.
\smallskip

(ia) Suppose first that $\psi\in\Irr(\bL^F)$ is unipotent with
$\lev(\psi)=2$. Then $\psi$ must be the tensor product of
$\lambda_1$ or $\lambda_2$ with the trivial character of the torus
$T=\bZ(\bL^F)$ of order $q^2-1$. 
Then
\[
  R_\bL^\bG(\psi)=(\tw2B_2[1,3]:2)+(\tw2B_2[1,3]:1^2) \quad\text{or}\quad
  (\tw2B_2[1,5]:2)+(\tw2B_2[1,5]:1^2),
\]
in \cite{Chevie}. These constituents of $R_\bL^\bG(\psi)$ are unipotent
characters of $\tw2F_4(q^2)$ whose field of values is $\QQ(\sqrt{-1})$. Indeed,
a quick inspection of the known character table of $\tw2F_4(q^2)$ (see
\cite{Malle90}) reveals that the values of $(\tw2B_2[1,3]:2)$ and
$(\tw2B_2[1,5]:2)$ at a unipotent element with centralizer of order $4q^8$ are
$-q/\sqrt{2}\pm q^2\sqrt{-1}$, whereas both $(\tw2B_2[1,3]:1^2)$ and
$(\tw2B_2[1,5]:1^2)$ vanish at that element. It follows that the field of
values of $R_\bL^\bG(\psi)$ is $\QQ(\sqrt{-1})$, and hence its $2$-rationality
level is~2, as required.
\smallskip

(ib) Next suppose that $\psi\in\Irr(\bL^F)$ is the tensor product of
$\lambda_1$ or $\lambda_2$ with a nontrivial irreducible character,
say $\eta$, of the torus $T$. We identify both $\bL^F$ and $\bG^F$
with their dual groups. Viewing $\eta$ as a linear character of
$\bL^F$, let $s\in \bT^F$ correspond to $\eta$ under the natural
bijection between the group of all linear characters of $\bL^F$ and
$\bT^F$ (see \cite[Proposition~2.5.20]{GM20}). Then
$\bC_{\bG^F}(s)=\bL^F$, and $\psi$ lies in the Lusztig series
$\cE(\bL^F,s)$ labeled by $s$. By \cite[Theorem~3.3.22]{GM20},
Lusztig induction induces a bijection from $\cE(\bL^F,s)$ to
$\cE(\bG^F,s)$. In particular, $R_\bL^\bG(\psi)\in\Irr(\bG^F)$.
Moreover, the $\bG^F$-conjugacy class of an element of $T$ is the
orbit under inversion, so the characters $R_\bL^\bG(\psi)$ in
question are precisely $q^2-2$ irreducible characters of
$\tw2F_4(q^2)$ of degree $|\bG^F:\bL^F|_{r'}\,\lambda_i(1)
=(q/\sqrt{2})(q^2-1)(q^2+1)^2(q^4-q^2+1)(q^{12}+1)$. All of them
have $2$-rationality level $2$, and we are done as well.
\medskip

(ii) Consider the Levi subgroup $\bL^F\cong\tw2B_2(q^2)\times
C_{q^2+\sqrt{2}q+1}$.
\smallskip

(iia) Suppose first that $\psi\in\Irr(\bL^F)$ is unipotent with
$\lev(\psi)=2$. As before $\psi$ must be the tensor product of
$\lambda_1$ or $\lambda_2$ with the trivial character of the torus
$T_1=\bZ(\bL^F)$, of order $q^2+\sqrt{2}q+1$. Its Lusztig induction
$R_\bL^\bG(\psi)$ decomposes as
\[\begin{aligned}
  -\phi_{2,1}+\phi_{1,4}''&+\phi_{1,4}'+(\tw2B_2[1,3]:2)-(\tw2B_2[1,3]:1^2)\\
  &-\tw2F_4[-I]+\tw2F_4^2[I]-\tw2F_4^3[-1]+\tw2F_4^4[-1],
\end{aligned}\]
respectively,
\[\begin{aligned}
  -\phi_{2,1}+\phi_{1,4}''+&\phi_{1,4}'+(\tw2B_2[1,5]:2)-(\tw2B_2[1,5]:1^2)\\
  &-\tw2F_4[I]+\tw2F_4^2[-I]-\tw2F_4^3[-1]+\tw2F_4^4[-1],
\end{aligned}\]
again in the notation of \cite{Chevie}. Checking the character table of
$\tw2F_4(q^2)$, we observe that the value of $R_\bL^\bG(\psi)$ at a unipotent
element with centralizer of order $2q^{14}(q^2-1)(q^4+1)$ is irrational.
Therefore $\sqrt{-1}\in\QQ(R_\bL^\bG(\psi))$ and we are done in this case.
\smallskip

(iib) Next consider the case where $\psi\in\Irr(\bL^F)$ is the
tensor product of $\lambda_1$ or $\lambda_2$ with a nontrivial
irreducible character of the torus $T_1$. As in (ib), we have
$R_\bL^\bG(\psi)\in\Irr(\bG^F)$. Here, the relative Weyl group of
$T_1$ identifies with the Galois group
$\Gal(\FF_{q^8}/\FF_{q^2})\cong C_4$ and each $\bG^F$-conjugacy
classes of nontrivial elements of $T_1$ consists of exactly four
elements. So the Lusztig induced characters $R_\bL^\bG(\psi)$ in
question are indeed $(q^2+\sqrt{2}q)/2$ irreducible characters of
$\tw2F_4(q^2)$ of degree $|\bG^F:\bL^F|_{r'}\,\lambda_i(1)
=(q/\sqrt{2})(q^2-\sqrt{2}q+1)(q^4-1)^2(q^4-q^2+1)(q^8-q^4+1)$.
Again, they all have $2$-rationality level $2$, and we are done.
\medskip

(iii) Finally we note that when $q^2=2$, the Levi subgroup
$\bL^F\cong \tw2B_2(q^2)\times C_{q^2-\sqrt{2}q+1}$ coincides with
$\tw2B_2(q^2)\times C_{q^2-1}$, which has already been treated
above.
\end{proof}

\begin{remark}
We remark that the works of Geck~\cite{G03} and Tiep--Zalesskii~\cite{TZ04},
which rely on Lusztig's theory of character sheaves and on strong
rationality results for unipotent elements of finite reductive groups $\bG^F$,
respectively, imply that for defining characteristic $r>2$ all irreducible
characters of $\bG^F$ are almost $r$-rational. This is no longer true when
$r=2$, even if one restricts to characters of odd degree, as observed above.
\end{remark}

\subsection{A conjecture of Navarro and Tiep}
In the introduction, we mentioned Navarro--Tiep's conjecture
\cite[Conjecture~C]{Navarro-Tiep21} concerning the field of values
of Sylow restrictions. The conjecture asserts that if $\chi$ is a
$p'$-degree irreducible character of a finite group $G$ and
$P\in\Syl_p(G)$, then
\[ \QQ(\ze_p)(\Res_P^G(\chi))=\QQ(\ze_{\lev(\chi)}). \]
Of course, this is immediate if $\chi$ is almost $p$-rational. When
$\lev(\chi)\ge 2$, this equality is equivalent to
$\lev(\chi)=\lev(\chi_P)$ if $p$ is odd and, on the other hand, it
is equivalent to the combination of $\lev(\chi)=\lev(\chi_P)$ and
$\zeta_4\in\QQ(\chi_P)$ if $p=2$.

Here, we take the opportunity to prove the conjecture in the case
where $G$ is a finite reductive group and $p$ is its defining
characteristic. This provides additional evidence supporting the
conjecture, complementing the results already presented in
\cite[\S7]{Navarro-Tiep21}.

\begin{theorem}
 Let $\bG^F$ be a finite reductive group defined in characteristic~$r$. Let
 $P\in\Syl_r(\bG^F)$ and $\chi\in\Irr_{r'}(\bG^F)$. Then
 $\QQ(\ze_r)(\Res_P^{\bG^F}(\chi))=\QQ(\ze_{\lev(\chi)})$.
\end{theorem}

\begin{proof}
We use the notation and setup from the beginning of this subsection, but for
$(\bG,F)$ instead of $(\bL,F)$. In particular, if $\pi$ denotes the central
isogeny from $\bG_{sc}$ to $[\bG,\bG]$, then its restrictions to $\wt{\bG}_j$
induce corresponding isogenies $\wt{\bG}_j \to \bG_j$, and
\[ M:= \pi(\bG_{sc}^F) = \prod_{j\in \cJ}\pi(\wt{\bG}_j^{F^{l_j}}) \]
is a normal subgroup of $\bG^F$ of $r'$-index.

We may assume that $\lev(\chi)\ge 2$. By Lemma~\ref{lem:p=r}, we have
$\lev(\chi)=2$, and $\bG$ has an $F$-stable simple component among the $\bG_j$,
say $\bG_1$, such that
\[ \bG_1^F=\wt{\bG}_1^F\in \{\tw2B_2(2), \tw2F_4(2)\}. \]
We identify this component as follows. Let $\phi\in \Irr(M)$ be a
character lying under $\chi$. Then $\lev(\phi)=2$. Since $\phi$ is a
tensor product of irreducible odd-degree characters of the groups
$\pi(\wt{\bG}_j^{F^{l_j}})$ ($j\in \cJ$), one of
its factors must have level $2$. By Lemma~\ref{lem:Malle}, the
corresponding factor group is either $\tw2B_2(2)$ or
$\tw2F_4(2)$. We denote the $\pi(\bG_1^F)$-component of $\phi$ by
$\lambda$, so that $\lev(\lambda)=2$.

Note that $\tw2B_2(2)$ is a Frobenius group with kernel $K\cong
C_5$ and complement $H\cong C_4$, while $\tw2F_4(2)$ has trivial
center. Therefore
\[ M_1:=\pi(\bG_1^F)\cong \tw2B_2(2),\; C_4,\; \text{or } \tw2F_4(2). \]
Furthermore, $\tw2B_2(2)$ has two odd-degree irreducible
characters (of degree $1$) of level $2$, each with kernel exactly
$K$. Their fields of values are $\QQ(\ze_4)$, both as characters of
$\tw2B_2(2)$ and of $\tw2B_2(2)/K\cong H$. Similarly,
$\tw2F_4(2)$ has eight odd-degree irreducible characters (of
degrees $27$ and $351$) of level $2$, all with field of values
$\QQ(i)$. In particular, $\lambda\in \Irr(M_1)$ is such a character.
Moreover, there exists a unipotent element $u_1\in M_1$ such that
$\lambda(u_1)=\ze_4$, and for other elements of $M_1$ of the same order
as $u_1$, the value of $\lambda$ is either $\ze_4$ or $-\ze_4$.

Let $u \in M$ be the unipotent element whose projection to the
direct factor $M_1$ is $u_1$ and to all other factors is the
identity. Then
\[ \phi(u)=a\lambda(u_1)=a\ze_4, \]
where $a:=\phi(1)/\lambda(1)\in\ZZ_{\ge1}$. In fact, by the preceding paragraph
and the fact that each $\pi(\wt{\bG}_j^{F^{l_j}})$ is normal in $\bG^F$, it
follows that for every $\bG^F$-conjugate $u'$ of $u$ we have
\[ \phi(u')\in\{a\ze_4,-a\ze_4\}. \]

Let $\phi_1=\phi,\phi_2,\dots,\phi_t$ be the distinct
$\bG^F$-conjugates of $\phi$. By Clifford theory,
\[ \chi(u)=[\Res_M^{\bG^F}(\chi),\phi]\sum_{k=1}^t \phi_k(u). \]
Each $\phi_k(u)$ is the value of $\phi$ at a $\bG^F$-conjugate of
$u$, and hence equals either $a\ze_4$ or $-a\ze_4$, as noted above.
Moreover, $t$ divides $|\bG^F:M|$, which is odd. It follows that
$\chi(u)$ is a \emph{nonzero} integer multiple of $i$.

We may of course choose the Sylow $r$-subgroup $P$ so that $u\in P$.
Hence
\[ i\in \QQ(\Res_P^{\bG^F}(\chi)). \]
On the other hand,
\[
\QQ(\Res_P^{\bG^F}(\chi)) \subseteq \QQ(\ze_{|\bG^F|_2}) \cap
\QQ(\chi)\subseteq\QQ(\ze_{|\bG^F|_2}) \cap
\QQ(\ze_{4|\bG^F|_{2'}}) = \QQ(\ze_4).
\]
Therefore $\QQ(\Res_P^{\bG^F}(\chi))=\QQ(\ze_4)$, completing the
proof.
\end{proof}

\subsection{The case of $p$-solvable Levi and $F$-stable
parabolic}\label{subsec:F-stable-parabolic}


The following lemma is a consequence of Isaacs--Navarro's proof
\cite{IN} of \cite[Conjecture~C]{Navarro-Tiep21} that was considered
in the previous subsection.

\begin{lemma}   \label{lem:p-solvable}
 Let $p$ be a prime. Let $M$ be a $p$-solvable subgroup of $p'$-index of a
 group $G$ and $P\in\Syl_p(M)$. Let $\psi\in\Irr(M)$ of $p'$-degree with
 $\lev(\psi)\ge 2$, and set $\chi:=\Ind_M^G(\psi)$.
 Then $\lev(\psi)=\lev(\chi)$.
\end{lemma}

\begin{proof}
As before we write $|G|=p^{a}m$, where $a\in\ZZ_{\ge 0}$ and
$m\in\ZZ_{\ge 1}$ is coprime to $p$. Under the assumptions that $M$
is $p$-solvable and that $\psi$ has $p'$-degree, using
\cite[Theorem~2.2]{IN}, we have
\[
\QQ(\ze_{mp})(\psi) =
\QQ\!\left(\ze_{mp},\,\ze_{p^{\al}}\right),
\]
where \[\al:=\ell\left(\Res_P^G(\psi)\right).\]

Recall that, for an abelian extension $\KK$ of $\QQ$ we write $c(\KK)$ for
the smallest positive integer $n$ such that $\KK\subseteq
\QQ(\ze_n)$, and we define $\lev(\KK):=\nu_p\bigl(c(\KK)\bigr)$.
Since $\lev(\psi)\ge 2$, we obtain
\[ \lev\left(\QQ(\ze_{mp})(\psi)\right)=\lev(\psi). \]
On the other hand,
\[
\lev\!\left(\QQ\!\left(\ze_{mp},\,\ze_{p^{\al}}\right)\right)
=\max\left\{1,\al\right\}.
\]
It follows that $\al\ge 2$, and indeed
$\lev(\psi)=\al=\ell\left(\Res_P^G(\psi)\right)$. The lemma now
follows from Lemma~\ref{lem:1}(i).
\end{proof}

\begin{theorem}   \label{thm:p-solvable}
 Conjecture~\ref{conj:Lusztig-induction} holds true in the case the
 parabolic subgroup $\bP$ is $F$-stable and $\bL^F$ is $p$-solvable.
\end{theorem}

\begin{proof}
When the parabolic subgroup $\bP$ is $F$-stable, one has the Levi
decomposition
\[ \bP^F = \bU^F \rtimes \bL^F, \]
where $\bU$ is the unipotent radical of $\bP$. In this situation,
Lusztig induction coincides with Harish-Chandra induction from
$\bL^F$ to $\bG^F$ (see \cite[Proposition~3.3.3]{GM20}). More
precisely,
\[ \chi := R_{\bL\subseteq\bP}^{\bG}(\psi) \]
is obtained from $\psi$ by first inflating it to $\bP^F$ via the
isomorphism $\bL^F \cong \bP^F/\bU^F$, and then inducing to $\bG^F$.

Let $\varphi$ denote the inflation of $\psi$ to $\bP^F$. Clearly,
$\varphi$ has the same field of values as $\psi$, and hence the same
$p$-rationality level:
\[ \lev(\psi) = \lev(\varphi). \]

If $\lev(\psi) \le 1$, then both $\psi$ and $\chi$ are almost
$p$-rational, and we are done. If $p=r$ then we are also done by
Theorem~\ref{thm:p=r}. Thus we may assume that $\lev(\psi) \ge 2$
and $p\neq r$. Note that
$\chi(1)=\pm|\bG^F:\bL^F|_{r'}\psi(1)=\pm|\bG^F : \bP^F|\psi(1)$.
The assumption on $\chi$ being of $p'$-degree implies that $\psi$
also has $p'$-degree and $|\bG^F : \bP^F|$ is prime to $p$. By
Lemma~\ref{lem:p-solvable}, it follows that
\[ \lev(\varphi) = \lev(\chi). \]
We have shown that $\lev(\psi) = \lev(\chi)$, as required.
\end{proof}

At this time we are not able to determine whether the conclusion of
Lemma~\ref{lem:p-solvable} remains valid without the assumption that
$M$ is $p$-solvable. If this were the case, then the proof of
Theorem~\ref{thm:p-solvable} would show that
Conjecture~\ref{conj:Lusztig-induction} holds when the parabolic
subgroup $\bP$ is $F$-stable.


\end{document}